\documentclass[12pt]{amsart}

\usepackage{tgpagella}
\usepackage{euler}
\usepackage[T1]{fontenc}
\usepackage{amsmath, amssymb}
\usepackage[hidelinks]{hyperref}
\usepackage[english]{babel}
\usepackage{mathrsfs}
\usepackage{eucal}
\usepackage[all]{xy}
\usepackage{tikz}
\usepackage{xparse}

\usepackage{todonotes}

\newtheorem{thm}{Theorem}[section]
\newtheorem*{thm*}{Theorem}
\newtheorem{lem}[thm]{Lemma}
\newtheorem{fact}[thm]{Fact}

\newtheorem{prop}[thm]{Proposition}
\newtheorem*{prop*}{Proposition}

\newtheorem{cor}[thm]{Corollary}
\newtheorem*{cor*}{Corollary}

\theoremstyle{definition}
\newtheorem{defn}[thm]{Definition}
\newtheorem*{defn*}{Definition}

\newtheorem{remark}[thm]{Remark}

\newtheorem{question}[thm]{Question}
\newtheorem{example}[thm]{Example}

\newtheorem*{question*}{Question}
\newtheorem*{Pquestion*}{Popa's question}

\newtheorem*{conv*}{Convention}

\def\bb{\mathbb}

\def\bb{\mathbb}

\def\cal{\mathcal}

\makeatletter

\def\dotminussym#1#2{%
  \setbox0=\hbox{$\m@th#1-$}%
  \kern.5\wd0%
  \hbox to 0pt{\hss\hbox{$\m@th#1-$}\hss}%
  \raise.6\ht0\hbox to 0pt{\hss$\m@th#1.$\hss}%
  \kern.5\wd0}

\DeclareMathOperator{\tr}{tr}

\newcommand\bN{{\mathbb N}}

\def \d{\mathbf{D}}

\newcommand{\norm}[1]{{\left\lVert #1\right\rVert}}
\newcommand{\onorm}[1]{\left\lVert #1 \right\rVert_{\infty}}
\newcommand{\tnorm}[1]{\left\lVert #1 \right\rVert_{2}}

\NewDocumentCommand{\cgen}{m g}{%
  \IfNoValueTF{#2}
  {\mathrm{C}^{\mbox{$*$}} \langle #1 \rangle}
  {\mathrm{C}^{\mbox{$*$}} \langle #1 : #2 \rangle}
}%

\title{Effective subfactor theory}
\author{Alec Fox and Isaac Goldbring}
\thanks{Goldbring was partially supported by NSF grant DMS-2054477.}

\address{Department of Mathematics\\University of California, Irvine, 340 Rowland Hall (Bldg.\# 400),
Irvine, CA 92697-3875}
\email{isaac@math.uci.edu}
\urladdr{http://www.math.uci.edu/~isaac}

\begin{document}

 \begin{abstract}
We characterize when a subfactor $N\subseteq M$ is oracle computable relative to a presentation of the ambient factor $M$ in terms of computability of the Jones basic construction, in terms of computable Pismner-Popa bases, and in terms of computability of the conditional expectation map.  We illustrate our main theorem with some examples.
 \end{abstract}
\maketitle

\section{Introduction}

This paper contributes to the recent literature on effective structure theory as it applies to operator algebras \cite{foxpres, KEP, gold, universal, hyper}  In particular, we study in a more systematic manner the question first raised in \cite{gold}, namely the complexity (in the sense of computability theory) of a subfactor inclusion.

More precisely, suppose that $N\subseteq M$ is a finite-index inclusion of II$_1$ factors\footnote{In this note, all II$_1$ factors are assumed to be separably acting.}; the rough question we wish to study is how much more complicated is the inclusion $N\subseteq M$ than just the complexity contained in $M$ itself?  In this paper, we introduce three natural notions of what it means for the inclusion $N\subseteq M$ to be no more complicated than $M$ itself and show that these three notions coincide.

Before explaining these three approaches to the complexity of a subfactor inclusion $N\subseteq M$, we must briefly explain how we are to compare its complexity to the complexity of the ambient factor $M$.  
First, the entire discussion is relative to a given \textbf{presentation} $M^\#$ of $M$, which is simply a countable sequence from the unit ball of $M$ that generates a SOT-dense $*$-subalgebra of $M$.  
With a presentation $M^\#$ of $M$ in hand, one can then reason about $M$ from an algorithmic viewpoint.
The most immediate question is how complicated (in the sense of Turing degree) must an algorithm be in order to calculate the $*$-algebra operations and the $2$-norm of $M$ from the presentation $M^\#$.
In general, we do not expect these operations to be computable (although this is the case for the standard presentation of the hyperfinite II$_1$ factor $\cal R$).  Instead, our standing assumption is that the presentation $M^\#$ of $M$ is $\d$-computable, where $\d$ is some (Turing) oracle that is knowledgeable enough to be able to carry out the aforementioned calculations.

We can now rephrase our main question:  if the presentation $M^\#$ of $M$ is $\d$-computable, what does it mean to say that the subfactor inclusion $N\subseteq M$ is also $\d$-computable?

One reasonable interpretation of this question involves the basic construction of Jones, which is the II$_1$ factor $M_1$ generated by $M$ (viewed as concretely represented on $L^2(M)$ via its standard representation) and the orthogonal projection $e_N:L^2(M)\to L^2(N)$.  The presentation $M^\#$ of $M$ induces a natural presentation $M_1^\#$ of $M_1$ obtained by simply adding $e_N$ to the presentation.  One can then interpret our rough question from above as:  when is $M_1^\#$ also $\d$-computable?  (From the definition, $M_1^\#$ must be at least as complicated as $M^\#$.)

Another perspective on the question comes from considering Pimsner-Popa bases for the inclusion.  A Pimsner-Popa basis for the inclusion $N\subseteq M$ leads to an elegant description of $M$ as a finitely generated projective right $N$-module.  It is thus reasonable to suspect that the subfactor is $\d$-computable if it has a $\d$-computable (with respect to $M^\#$) Pismner-Popa basis and such that the inclusion of $N^\dagger$ into $M^\#$ is $\d$-computable with respect to some presentation $N^\dagger$ of $N$. Our main theorem asserts that (given a precise notion of $\d$-computable Pimsner-Popa basis), this notion coincides with the $\d$-computability of $M_1^\#$.

One final perspective on the $\d$-computability of subfactors comes from considering the $\d$-computability of the conditional expectation $E_N:M^\#\to M^\#$.  It turns out that we were not able to show that this is equivalent to the previous two notions on its own but is equivalent to the previous two notions when combined with the assumption that the index $[M:N]$ is a $\d$-computable real number.  In certain instances (such as when $[M:N]\leq 4$ or when $M$ has finite depth), the index is simply computable and thus the $\d$-computability of $E_N$ is equivalent to the previous two notions.  We consider some examples of our main theorem in action in the final section of the paper.

It follows from our main result that if $M_1^\#$ is $\d$-computable, then so are the iterated basic constructions $M_n^\#$ comprising the Jones tower.

Two of the three equivalent conditions in our main theorem involve a conjunction of two statements.  It is natural to wonder if either of those conjunctions can be simplified to one of their conjuncts.  We conclude the paper with two partial results along these lines.  We show that if $E_N$ is $\d$-computable, then $M_1^\#$ is $\d'$-computable (so $M_1^\#$ is at most one ``jump'' of complexity higher than $M^\#$).  We also show that if there is some presentation $N^\dagger$ of $N$ such that the inclusion $N^\dagger\hookrightarrow M^\#$ is $\d$-computable, then $M_1^\#$ is $\d''$-computable (so at most two ``jumps'' of complexity higher).  It would be interesting to see if these jumps are actually necessary.

In future work, we plan on studying effective aspects of other objects related to a subfactor, such as its standard invariant and derived tower.

We assume that the reader is familiar with basic facts about tracial von Neumann algebras but not necessarily basic subfactor theory; we present what we need about subfactor theory in Section 4.  We also assume that the reader is familiar with very basic computability theory; the uninitiated reader can consult \cite[Section 2]{gold}, where it was presented with an eye towards the issues that concern us here.  Nevertheless, we do discuss the basic notions concerning presentations of tracial von Neumann algebras in the next section.

We thank David Penneys for many useful discussions regarding this work.

\section{Generalities on Presentations}

In this section, we discuss the notion of a presentation of a tracial von Neumann algebra and collect some basic facts to be used throughout the paper.

\begin{defn}
Given a tracial von Neumann algebra $M$, a \textbf{presentation} of $M$ is a countable sequence $(a_n)_{n\in \bb N}$ of elements of the unit ball of $M$ that generate a (SOT-) dense *-subalgebra of $M$.  The elements of the presentation are sometimes called the \textbf{special points} of the presentation.
\end{defn}

We use letters like $M^\#$ and $M^\dagger$ to denote presentations of $M$.  

\begin{defn}
Given a presentation $M^\#$ of a tracial von Neumann algebra, the set of elements of $M$ obtained by closing the elements of the presentation under addition, multiplication, adjoint, and multiplication by elements of $\bb Q(i)$ are referred to as the \textbf{rational points} of $M^\#$.
\end{defn}

Throughout, we let $\cal F(k) = \cgen{c_1,\ldots,c_k}{\norm{c_i} \leq 1}$ be the universal contraction C*-algebra on $k$ generators.  

Given a presentation $M^\#$ of $M$, there is an associated computable enumeration $(x_n)_{n\in \bb N}$ of the rational points of the presentation.  
Moreover, there is a computable functional $\cdot^\flat$ on the rational points which computes a na\"ive upper bound on the operator norm.
Specifically, if $x_n = p(a_1,\ldots,a_k)$ is a rational point (so $a_1,\ldots,a_k$ are special points of $M^\#$ and $p(w_1,\ldots,w_k)$ is a $*$-polynomial with coefficients from $\bb Q(i)$), then we set \begin{align*}
x_n^\flat &:= \|p(c_1,\ldots,c_k)\|_{\cal F(k)}\\
&=\sup\{\norm{p(z_1,\ldots,z_k)}_A : A \text{ is a C*-algebra}, \norm{z_i}_A \leq 1\} .
\end{align*}
In particular, $a_n^\flat = 1$ for every special point $a_n$ of the presentation.
Since the standard presentation of $\cal F(k)$ is computable by \cite[Corollary 3.8]{foxpres}, it follows that $x_n^\flat$ is a computable real number, uniformly in $n$.

A key consequence of our set-up is the following:

\begin{lem}\label{flatrationalpoints}
    For any rational point $x$ of $M^\#$ with $\onorm{x} \leq 1$, there is a sequence $(x_n)_{n\in \bb N}$ of rational points of $M^\#$ such that $\onorm{x_n - x} \to 0$ with $x_n^\flat < 1$ for all $n$.
\end{lem}
\begin{proof}
    Let $x = p(a_1,\ldots,a_k)$.
    Let $\pi : \cal{F}(k) \to \cgen{a_1,\ldots,a_k} \subseteq M$ be the natural surjection defined by $\pi(c_i) = a_i$.
    By \cite[2.2.10]{rordamktheory}, there exists $z \in \cal{F}(k)$ such that $\norm{z} = \onorm{x} \leq 1$ and $\pi(z) = x$.
    Let $(z_n)$ be a sequence of rational *-polynomials in $c_1,\ldots,c_k$ such that $\norm{z_n - z} \to 0$ and $\norm{z_n} < 1$ for all $n$.
    Let $x_n = \pi(z_n)$, so $x_n$ is a rational point of $M^\#$ and $x_n^\flat = \norm{z_n} < 1$.
    It remains to note that $\onorm{x_n - x} = \onorm{\pi(z_n - z)} \leq \norm{z_n - z} \to 0$.
\end{proof}

\begin{cor}\label{flatfund}
For any $x\in M$ with $\onorm{x} \leq 1$, there is a sequence $(x_n)_{n\in \bb N}$ of rational points of $M^\#$ such that $\tnorm{x_n - x} \to 0$ with $x_n^\flat < 1$ for all $n$.
\end{cor}

\begin{proof}
By Kaplansky density, there is a sequence of rational points $y_n$ converging to $x$ in the $2$-norm, each with $\onorm{y_n}\leq 1$. 
For each rational point $y_n$, by Lemma \ref{flatrationalpoints}, we can choose a rational point $x_n$ close enough with $x_n^\flat < 1$.
\end{proof}

Throughout this paper, $\d$ denotes a Turing oracle.

\begin{defn}
If $M^\#$ is a presentation of $M$, then $x\in M$ is a \textbf{$\d$-computable point of $M^\#$} if there is a $\d$-algorithm such that, upon input $k\in \mathbb N$, returns a rational point $p\in M^\#$ with $\|x-p\|_2<2^{-k}$.
\end{defn}

The following lemma is clear.

\begin{lem}
The $\bf D$-computable points of $M^\#$ are closed under addition, multiplication, adjoint, and multiplication by rational complex scalars.
\end{lem}

\begin{defn}
If $M^\#$ is a presentation of $M$ and $\d$ is an oracle, then $M^\#$ is a \textbf{$\d$-computable presentation} if there is a $\d$-algorithm such that, upon input rational point $p\in M^\#$ and $k\in \mathbb N$, returns a rational number $q$ such that $|\|p\|_2-q|<2^{-k}$.
\end{defn}

If $M^\#$ is a $\d$-computable presentation, we refer to any code for a $\d$-algorithm as in the definition as a \textbf{code for the presentation $M^\#$}.  

Given a presentation $M^\#$ of $M$ whose special points are enumerated by $(a_n)_{n\in \bb N}$ and a projection $p\in M$, we let $(pMp)^\#$ denote the canonical induced presentation of the compression $pMp$ whose special points are $(pa_np)_{n\in \bb N}$.  Note that there is a computable function which, upon input a code for a $\d$-computable presentation $M^\#$ and a code for a $\d$-computable projection $p$ in $M^\#$, returns a code for the canonical presentation $(pMp)^\#$.  Note also that if $M^\#$ and $p$ are $\d$-computable, then every $\d$-computable point of $(pMp)^\#$ is a $\d$-computable point of $M^\#$, uniformly in the code for the presentation and a code for the point. 

The following lemma is clear:

\begin{lem}\label{cornerlemma}
There is a computable function $f_{corner}:\bb N^2\to \bb N$ such that if $x$ is the code of a $\d$-computable presentation $M^\#$ of a tracial von Neumann algebra $M$ and $y$ is the code of a $\d$-computable projection of $M^\#$, then $f_{corner}(x,y)$ is the code of the canonical $\d$-computable presentation $(pMp)^\#$ of $pMp$.
\end{lem}

\begin{defn}
Suppose that $M$ and $N$ are tracial von Neumann algebras with presentations $M^\#$ and $N^\dagger$ respectively.  Further suppose that $f:M^m\to N$ is a Lipshitz map (say with respect to the maximum metric on $M^m$)\footnote{The Lipshitz condition can be weakened to having a ``computable modulus of uniform continuity'' but we will not need this more general notion in this paper.}.  Then $f$ is a \textbf{$\d$-computable map from $M^\#$ into $N^\dagger$} if there is a $\d$-algorithm such that, upon input a tuple of rational points $\vec p\in (M^\#)^m$ and $k\in \mathbb N$, returns a rational point $p'\in N^\dagger$ such that $\|f(\vec p) - p'\|<2^{-k}$.  (Here, we use an effective numbering of $\mathbb N^m$ to effectively enumerate the $m$-tuples of rational points of $M^\#$.)
\end{defn}

Throughout the paper, $\bb C$ always has its standard presentation, that is, with $\bb Q(i)$ as its presentation.

\section{Useful technical lemmas}

In this section, we collect a number of useful technical lemmas stating that certain types of elements (such as projections and partial isometries) can, under certain assumptions, be found effectively.

\begin{fact}\label{nearprojection}
Suppose $0 < \epsilon < 1$ and $0 < \delta < \epsilon^2/48$.
Then for any tracial von Neumann algebra $M$ and $x\in M$, if $\onorm{x} \leq 1$, $\tnorm{x - x^*} \leq \delta$, and $\tnorm{x - x^2} \leq \delta$, then there is a projection $p\in M$ such that $\tnorm{x-p} < \epsilon$.
\end{fact}
\begin{proof}
Let $z = x^*x$.
Then $0 \leq z \leq 1$ and \[\tnorm{x - z} \leq \tnorm{x - x^2} + \onorm{x}\tnorm{x - x^*} \leq 2\delta.\]
Hence, \begin{align*}
\tnorm{z^2 - z} &\leq \tnorm{z^2 - zx} + \tnorm{zx - z} \\
&\leq \onorm{z}\tnorm{z - x} + \onorm{x^*}\tnorm{x^2 - x} \\
&< \epsilon^2/16.
\end{align*}
Thus, by \cite[Lemma 1.1.5]{connes}, there is a projection $p$ such that $\tnorm{z - p} \leq \epsilon/2$, whence $\tnorm{x - p} < \epsilon$.
\end{proof}
 
\begin{defn}
    Given a presentation $M^\#$ of a tracial von Neumann algebra $M$ and rational $\epsilon\in (0,1)$, we say a rational point $x$ of $M^\#$ is an \textbf{$\epsilon$-quasi-projection} if $x^\flat < 1$, $\tnorm{x - x^*} < \epsilon^2/48$, and $\tnorm{x - x^2} < \epsilon^2/48$.
\end{defn}

\begin{lem}\label{quasi}
Fix a presentation $M^\#$ of a tracial von Neumann algebra $M$ and rational $\epsilon \in (0,1)$.
    \begin{enumerate}
        \item Every $\epsilon$-quasi-projection is $\epsilon$-close (in $2$-norm) to an actual projection.
        \item Every projection in $M$ is $\epsilon$-close (in $2$-norm) to an $\epsilon$-quasi-projection.
        \item If $M^\#$ is $\d$-computable, then the set of $\epsilon$-quasi-projections of $M^\#$ is $\d$-c.e.
    \end{enumerate}
\end{lem}

\begin{proof}
Item (1) follows from Fact \ref{nearprojection} while item (2) follows from Corollary \ref{flatfund}.  Item (3) is clear.
\end{proof}

\begin{lem}\label{earliertracelemma}
Suppose that $M^\#$ is a $\bf D$-computable presentation of a tracial von Neumann algebra $M$. Then $(x,y) \mapsto \tr(x^*y):(M^\#)^2\to \bb C$ is a $\d$-computable map, uniformly in a code for $M^\#$.
\end{lem}

\begin{proof}
This follows from the polarization identity.
\end{proof}

\begin{lem}\label{identitylemma}
Suppose that $M^\#$ is a $\bf D$-computable presentation of a tracial von Neumann algebra $M$.  Then $1$ is a $\d$-computable point of $M^\#$, uniformly in a code for $M^\#$.
\end{lem}

\begin{proof}
Given $k \in \bN$, search for a rational point $x$ such that $xx^*x$ is a $2^{-2k-3}$-quasi-projection with $|\tr(xx^*x) - 1| < 2^{-2k-3}$, where we know $x \mapsto \tr(xx^*x)$ is $\d$-computable by Lemma \ref{earliertracelemma}.
Such a rational point must exist by Corollary \ref{flatfund} when applied to $1$.
By Fact \ref{nearprojection}, there is a projection $p$ with $\tnorm{xx^*x-p}< 2^{-2k-3}$.
Then 
\begin{align*}
    \tnorm{p - 1}^2 &= 1 - \tr(p)\\
    &\leq |1 - \tr(xx^*x)| + \tnorm{xx^*x - p} (\text{by Cauchy-Schwarz}) \\
    &< 2^{-2k-2}.
\end{align*}
Thus $\tnorm{xx^*x - 1} \leq \tnorm{xx^*x - p} + \tnorm{p - 1} < 2^{-k}$.
\end{proof}
By the previous lemma, if $M^\#$ is a $\d$-computable presentation of a tracial von Neumann algebra $M$, then we can compute, uniformly in a code for $M^\#$, a $\d$-computable presentation of $M$ with $1$ as a special point which is $\d$-computably isomorphic to $M^\#$, uniformly in $M^\#$.
As a result, with no loss of generality, \textbf{in the rest of the paper, we always assume that $1$ is in fact a special point of $M^\#$}.

\begin{cor}\label{tracelemma}
If $M^\#$ is a $\d$-computable presentation of tracial von Neumann algebra $M$, then $\tr:M^\#\to \bb C$ is a $\d$-computable map, uniformly in a code for $M^\#$.
\end{cor}

    

\begin{lem}\label{aprojlemma}
There is a computable function $f_{ap}:\bb N^3\to \bb N$ such that if $x$ is the code of a $\d$-computable presentation $M^\#$ of a II$_1$ factor $M$ and $r$ is the code of a $\d$-computable real number $\lambda \in [0,1]$, then $f_{ap}(x,r,k)$ is the code of a $\d$-computable projection of $M^\#$ with trace within $1/2^k$ of $\lambda$.
\end{lem}

\begin{proof}
    For $n \in \bN$, let $\epsilon_n = 2^{-n-k-2}$.
    We effectively determine a sequence $(x_n)_{n\in \bb N}$ of rational points such that $x_n$ is a $\epsilon_n$-quasi-projection, $\tnorm{x_{n+1} - x_n} < \epsilon_{n+1} + \epsilon_n $ and $|\tr(x_{n}) - \lambda| < \epsilon_{n} + 2\sum_{i=1}^{n-1} \epsilon_i$.
    
    First, search for the first $\epsilon_1$-quasi-projection $x_1$ such that $|\tr(x_1) - \lambda| < \epsilon_1$.
    Such a rational point must exist by Corollary \ref{flatfund} since $M$ has a projection with trace $\lambda$.
    
    Assuming $x_n$ has already been determined, search for the first $\epsilon_{n+1}$-quasi-projection $x_{n+1}$ such that $\tnorm{x_{n+1} - x_n} < \epsilon_{n+1} + \epsilon_n$ and $|\tr(x_{n+1}) - \lambda| < \epsilon_{n+1} + 2\sum_{i=1}^{n} \epsilon_i$.
    Such a rational point must exist by Corollary \ref{flatfund} since if $q$ is one of the projections $\epsilon_n$-close to $x_n$, then $|\tr(q) - \lambda| < \epsilon_n + \epsilon_n + 2\sum_{i=1}^{n-1} \epsilon_i$ by Cauchy-Schwarz, so any $\epsilon_{n+1}$-quasi-projection $\epsilon_{n+1}$-close to $q$ suffices.

    Now, observe that $(x_n)_{n\in \bb N}$ converges to a projection $p$ and $$|\tr(p) - \lambda| \leq 2\sum_{i=1}^{\infty} \epsilon_i < 2^{-k}.$$
\end{proof}

\begin{lem}\label{projlemma}
There is a computable function $f_{proj}:\bb N^2\to \bb N$ such that if $x$ is the code of a $\d$-computable presentation $M^\#$ of a II$_1$ factor $M$ and $r$ is the code of a $\d$-computable real number $\lambda \in (0,1)$, then $f_{proj}(x,r)$ is the code of a $\d$-computable projection of $M^\#$ with trace exactly $\lambda$.
\end{lem}

\begin{proof}

    We effectively determine a decreasing sequence $(p_n)_{n\in \bb N}$ of $\d$-computable projections of $M^\#$ such that $\tr_M(p_n) \in (\lambda, \lambda + 2^{-n})$.
    Set $p_0 = 1$.
    Assuming $p_n$ has already been determined, we find $p_{n+1}$.
    By Lemma \ref{cornerlemma}, we can effectively find a code $x_{p_n}$ for the $\d$-computable corner presentation $(p_nMp_n)^\#$.
    Find $k \in \bb{N}$ such that $k \geq n + 1$ and $\lambda/tr_M(p_n) + 2^{-k} < 1$.
    By Lemma \ref{aprojlemma}, we can effectively find a $\d$-computable projection $p_{n+1}$ of $(p_n M p_n)^\#$ such that $$\tr_{p_nMp_n}(p_{n+1}) \in (\lambda/\tr_M(p_n), \lambda/tr_M(p_n) + 2^{-k}),$$ namely the projection determined by $f_{ap}(x_{p_n}, \lambda/\tr_M(p_n) + 2^{-k-1}, k+1)$.
    Thus $p_{n+1} \leq p_n$,  $p_{n+1}$ is a $\d$-computable point of $M^\#$, and $\tr_M(p_{n+1}) \in (\lambda, \lambda + 2^{-(n+1)})$.

    Observe \[\tnorm{p_{n+1} - p_n}^2 = \tr(p_{n}) - \tr(p_{n+1}) < 2^{-n},\] so $(p_n)_{n\in \bb N}$ is a Cauchy sequence of projections, whose limit we denote by $p$.  Then $p$ is a projection, $\tr_M(p) = \lambda$, and $\norm{p_n - p}_2 = \sqrt{\tr(p_n) - \tr(p)} < 2^{-n/2}$.
    Therefore $p$ is a $\d$-computable projection of $M^\#$ of trace exactly $\lambda$.

\end{proof}

\begin{lem}\label{orthproj}
There is a computable function $f_{op}:\bb N^3\to \bb N$ such that if $x$ is the code of a $\d$-computable presentation of a II$_1$ factor $M$, $r$ is the code for a $\d$-computable real number $\lambda\in (0,1)$ and $n$ is $\lfloor 1/\lambda\rfloor$, then $f_{op}(x,r,n)$ is the code for a sequence $p_1,\ldots,p_{n+1}$ of mutually orthogonal $\d$-computable projections of $M^\#$ with $\tr(p_i)=\lambda$ for $i=1,\ldots,n$ and $\tr(p_{n+1})=1-n\lambda$.
\end{lem}

\begin{proof}
We effectively determine a sequence $(p_i)_{i=1}^n$ of mutually orthogonal $\d$-computable projections of $M^\#$ such that $\tr(p_i) = \lambda$.
By Lemma \ref{projlemma}, we can effectively find a $\d$-computable projection $p_1$ in $M^\#$ of trace $\lambda$.
Assuming $p_k$ has already been determined, we find $p_{k+1}$.
Let $q = 1 - \sum_{i=1}^{k} p_i$, so $q$ is a $\d$-computable projection with $\tr(q) \geq \lambda$ by Lemma \ref{identitylemma}.
By Lemma \ref{projlemma}, we can effectively find a $\d$-computable projection $p_{k+1}$ in $(qMq)^\#$ with $\tr_{qMq}(p_{k+1}) = \lambda / \tr(q)$.  It follows that $p_{k+1}$ is orthogonal to $p_i$ for $i=1,\ldots,k$ and $\tr(p_{k+1}) = \lambda$.
It remains to set $p_{n+1} = 1 - \sum_{i=1}^n p_i$.
\end{proof}

As usual, we denote the fact that projections $p$ and $q$ are Murray von Neumann equivalent by writing $p\sim q$, while the notation $p\preceq q$ signifies that $p\sim q'$ for some projection $q'\leq q$.
If $v$ is a partial isometry that satisfies $v^*v = p$ and $vv^* = q$, then we say that $v$ is an \textbf{implement} of $p \sim q$.

A proof of the non-effective form of the following lemma is essentially contained in the proof of \cite[Proposition 1.1.3(b)]{connes}.

\begin{lem}\label{almostMvNequivalent}
Let $M$ be a II$_1$ factor.
Let $p,q \in M$ be projections such that $p \sim q$.
Suppose $0 < \epsilon < 1$ and $0 < \delta < \epsilon^{16}/11^{16}$.
Then for any $x\in M$, if $\onorm{x} \leq 1$, $\tnorm{x^*x - p} \leq \delta$, and $\tnorm{xx^* - q} \leq \delta$, then there is an implement $v\in M$ of $p\sim q$ with $\tnorm{x - v} < \epsilon$.
\end{lem}
\begin{proof}
Let $x\in M$ be as in the conclusion of the theorem and set $y := qxp$.
Then \[\tnorm{(1-q)x}^2 = \tr((1-q)(xx^* - q)(1-q)) \leq \tnorm{xx^* - q} \leq \delta\] and \[\tnorm{x(1-p)}^2 = \tr((1-p)(x^*x - p)(1-p)) \leq \tnorm{x^*x - p} \leq \delta,\]
so \[\tnorm{x - y} \leq \tnorm{x - qx} + \tnorm{qx - y} \leq \tnorm{(1-q)x} + \tnorm{x(1-p)} \leq 2\delta^{1/2}.\]
Furthermore, \[\tnorm{y^*y - p} \leq \tnorm{y^*y - y^*x} + \tnorm{y^*x - x^*x} + \tnorm{x^*x - p} \leq 5\delta^{1/2},\]
whence it follows that \[\tnorm{(y^*y)^2 - y^*y} = \tnorm{(y^*y)^2 - y^*yp} \leq \tnorm{y^*y - p} \leq 5\delta^{1/2}.\]
Let $f : [0,1] \to [0,1]$ be defined by $f(a) = a^{-1/2}$ if $1 - \sqrt{5}\delta^{1/4} \leq a \leq 1$ and $0$ otherwise.
Set $z := yf(y^*y)$.
Then $z^*z$ and $zz^*$ are projections satisfying $$z^*z = y^*yf(y^*y)^2 \leq p \text{ and }zz^* = y(f^*y)^2y^* \leq q.$$
By construction, 
$(1-y^*y)^2(1 - z^*z) \geq (5\delta^{1/2})(1 - z^*z)$ and $\tr((y^*y)^2(1- y^*y)^2) \leq (5)^2\delta$,
whence \begin{align*}
\tnorm{y(1-z^*z)}^4 &= \tr(y^*y(1 - z^*z))^2\\
&\leq \tr((y^*y)^2(1 - z^*z))\\
&\leq(5\delta^{1/2})^{-1}\tr((y^*y)^2(1-y^*y)^2(1-z^*z)) \\
&\leq 5\delta^{1/2}.
\end{align*}
It follows that \begin{align*}
\tnorm{y - z} &\leq \tnorm{y - yz^*z} + \tnorm{yz^*z - z} \\
&\leq \tnorm{y(1-z^*z)} + \onorm{y^*yf(y^*y)^2 - f(y^*y)} \\
&\leq \sqrt[4]{5}\delta^{1/8} + \sqrt{5}\delta^{1/4}\\
&\leq 5\delta^{1/8}.
\end{align*}
Furthermore,
\[\tnorm{z^*z - p} \leq \tnorm{z^*z - z^*y} + \tnorm{z^*y - y^*y} + \tnorm{y^*y - p} \leq 15\delta^{1/8}.\]
Observe that $\tr(q) - \tr(zz^*) = \tr(p) - \tr(z^*z) = \tnorm{p - z^*z} \leq 15\delta^{1/8}$.
Let $w$ be a partial isometry such that $w^*w = p - z^*z$ and $ww^* = q - zz^*$.
Then $$\tnorm{w}^2 = \tr(p) - \tr(z^*z) \leq 15\delta^{1/8},$$ and so $\tnorm{w} \leq 4\delta^{1/16}$.
Finally, set $v := z + w$.
Then $v$ is an implement of $p\sim q$ and \[\tnorm{x - v} \leq \tnorm{x - y} + \tnorm{y - z} + \tnorm{w} \leq 11\delta^{1/16} < \epsilon.\]
\end{proof}

\begin{defn}
    Given a presentation $M^\#$ of a II$_1$ factor $M$, projections $p$ and $q$ in $M$ with $p \sim q$, and rational $\epsilon \in (0,1)$, we say that a rational point $x$ of $M^\#$ is an \textbf{$\epsilon$-quasi-implement of $p\sim q$} if $x^\flat < 1$, $\tnorm{x^*x - p} < \epsilon^{16}/11^{16}$, and $\tnorm{xx^* - q} < \epsilon^{16}/11^{16}$.
\end{defn}

\begin{lem}
Fix a presentation $M^\#$ of a II$_1$ factor $M$, projections $p$ and $q$ in $M$ with $p \sim q$, and rational $\epsilon \in (0,1)$.
    \begin{enumerate}
        \item Every $\epsilon$-quasi-implement of $p \sim q$ is $\epsilon$-close (in $2$-norm) to an actual implement of $p \sim q$.
        \item Every implement of $p \sim q$ is $\epsilon$-close (in $2$-norm) to an $\epsilon$-quasi-implement of $p\sim q$.
        \item If $M^\#$ is $\d$-computable and $p,q$ are $\d$-computable points of $M^\#$, then the set of $\epsilon$-quasi-implements of $p \sim q$ is $\d$-c.e.
    \end{enumerate}
\end{lem}

\begin{proof}
Item (1) follows from Lemma \ref{almostMvNequivalent} while item (2) follows from Corollary \ref{flatfund}.  Item (3) is clear.
\end{proof}

\begin{lem}\label{partialiso}
Suppose that $M^\#$ is a $\bf D$-computable presentation of a II$_1$ factor $M$ and $p,q$ are $\d$-computable projections of $M^\#$.
\begin{enumerate}
    \item If $p\sim q$, then there is a $\d$-computable implement of $p\sim q$,  which can furthermore be computed from codes for $M^\#$, $p$ and $q$.
    \item If $p\preceq q$, then there a $\d$-computable partial isometry $v$ of $M^\#$ such that $v^*v=p$ and $vv^*\leq q$, which can furthermore be computed from codes for $M^\#$, $p$ and $q$.
\end{enumerate}
\end{lem}

\begin{proof}

For (1), fix $n \in \mathbb{N}$ and set $\epsilon_n := 2^{-n}$.
We effectively determine a sequence $(x_n)_{n\in \bb N}$ of rational points of $M^\#$ such that $x_n$ is an $\epsilon_n$-quasi-implement of $p \sim q$ and $\tnorm{x_{n+1} - x_n} < \epsilon_{n+1} + \epsilon_n$.
First, search for an $\epsilon_1$-quasi-implement $x_1$ of $p \sim q$.
Assuming $x_n$ has already been determined, search for an $\epsilon_{n+1}$-quasi-implement $x_{n+1}$ of $p \sim q$ such that $\tnorm{x_{n+1} - x_n} < \epsilon_{n+1} + \epsilon_n$.
Such a rational point must exist since if $v$ is an implement of $p \sim q$ which is $\epsilon_n$-close to $x_n$, then any $\epsilon_{n+1}$-quasi-implement $x_{n+1}$ of $p \sim q$ which is $\epsilon_{n+1}$-close to $v$ suffices.  It remains to observe that $(x_n)_{n\in \bb N}$ converges to an implement of $p \sim q$.

To prove (2), by (1) it suffices to show that we can find a $\d$-computable projection $q'\leq q$ such that $p\sim q'$, or, equivalently, that $tr(q')=tr(p)$.  Since $p$ and $q$ are $\d$-computable elements of $M^\#$, $tr(p)$ and $tr(q)$ are $\d$-computable real numbers by Corollary \ref{tracelemma}, whence by Lemma \ref{projlemma}, we can find a $\d$-computable projection $q'\in (qMq)^\#$ whose corner trace is $tr(p)/tr(q)$, and thus $tr(q')=tr(p)$, as desired. 
\end{proof}



\section{Preliminaries on subfactor theory}

In this section, we present the basic facts from subfactor theory from \cite{jones,pimsnerpopa} needed in the rest of the paper.  
Another good reference for this material is \cite{popa}.

Throughout, by a \textbf{subfactor} we mean an inclusion of $\rm II_1$ factors $N\subseteq M$.  We let $e_N:L^2(M)\to L^2(N)$ (or simply $e$ if the context is clear) denote the canonical orthogonal projection map, noting that the restriction $E_N$ of $e_N$ to $M$ takes values in $N$; the linear map $E_N:M\to N$ is called the \textbf{conditional expectation map} from $M$ onto $N$.  We note the following properties of $E$:
\begin{itemize}
    \item $E_N(x)=x$ for all $x\in N$.
    \item $E_N(xyz)=xE_N(y)z$ for all $x,z\in N$ and $y\in M$.
    \item $E_N(x^*)=E_N(x)^*$ for all $x\in M$.
    \item $E_N(x)$ is the unique $y\in N$ such that $\tr((x-y)z)=0$ for all $z\in N$.
    \item $E_N(x)$ is the unique $y\in N$ such that $e_Nxe_N=ye_N$.
    \item $E_N$ is trace-preserving and contractive with respect to both the operator norm and the 2-norm.
\end{itemize}

An important invariant of the subfactor $N\subseteq M$ is the \textbf{index} $[M:N]$ of $N$ in $M$.  
In order to avoid introducing modules over a II$_1$ factor, we give a definition of the index in terms of the conditional expectation map due to Pimsner and Popa 
\cite{pimsnerpopa}, namely
$$[M:N]^{-1}:=\max\{\lambda \in \bb R_+: \ E_N(x)\geq \lambda x \text{ for all }x\geq 0\}.$$  This can alternatively be described by
$$[M:N]^{-1}:=\inf\left\{\frac{\|E_N(x)\|^2_2}{\|x\|^2_2} \ : \ x>0\right\}.$$

The terminology is inspired by the fact that if $L(H)\subseteq L(G)$ is the subfactor induced by a subgroup $H\leq G$, then $[L(G):L(H)]=[G:H]$.  However, unlike the case of groups, the index does not necessarily have to be an integer.  A landmark result of Jones states that $[M:N]$ must always be an element of the set $\{4\cos^2(\frac{\pi}{n}) \ : n=3,4,5,\ldots\}\cup [4,\infty]$.  Moreover, all of these values can be realized as indices of subfactors of $\cal R$.

Associated to the subfactor $N\subseteq M$ is the \textbf{Jones basic construction}, which is the von Neumann algebra $M_1:=( M\cup \{e_N\})''\subseteq \cal B(L^2(M))$ generated by $M$ and $e_N$.  We recall the following basic facts about $M_1$:

\begin{itemize}
    \item $M+Me_NM$ is a weakly dense $*$-subalgebra of $M_1$.
    \item $M_1$ is a II$_1$ factor if and only if $[M:N]<\infty$.  
\end{itemize}
Assuming $[M:N]<\infty$, we also have:
\begin{itemize}
    \item $[M_1:M]=[M:N]$.
    \item For $x\in M$, we have $\tr(xe_N)=[M:N]^{-1}\tr(x)$.  In particular, 
    $\tr(e_N)=[M:N]^{-1}$
    and
    $\tnorm{xe_N}^2=[M:N]^{-1}\tnorm{x}^2$
    for any $x\in M$.
    \item $E_M(e_N)=\tr(e_N)\cdot 1=[M:N]^{-1}\cdot 1$.
    \item For  $x\in M_1$, $y=[M:N]E_M(xe_N)$ is the unique element of $M$ such that $xe_N=ye_N$.  In particular, $M_1=Me_NM:=\operatorname{span}\{xe_Ny \ : \ x,y\in M\}$.
    \\\
\end{itemize}

Since $M\subseteq M_1$ is a subfactor, one can consider the associated basic construction, which is denoted $M_2$.  Continuing this process leads to the \textbf{Jones tower} $$N\subseteq M_0 := M\subseteq M_1\subseteq M_2\subseteq \cdots \subseteq M_n\subseteq M_{n+1}\subseteq \cdots.$$

Suppose that $[M:N]<\infty$ has integer part $n$.  Then there exist elements $m_1,\ldots,m_{n+1}$ of $M$ with the following properties:
\begin{itemize}
    \item $E_N(m_j^*m_k)=0$ for all distinct $j,k\in \{1,\ldots,n+1\}$.
    \item $E_N(m_j^*m_j)=1$ for all $j\in \{1,\ldots,n\}$.
    \item $E_N(m_{n+1}^*m_{n+1})$ is a projection of trace $[M:N]-n$.
\end{itemize}
(By convention, $m_{n+1}=0$ if $[M:N]\in\mathbb{N}$.)
Any such family of elements is called a \textbf{Pimsner-Popa basis of $M$ over $N$}.  The terminology is inspired by the fact that if $m_1,\ldots,m_{n+1}$ is a Pimsner-Popa basis of $M$ over $N$, then every $x\in M$ can be written uniquely as $x=\sum_{j=1}^{n+1}m_jx_j$ with each $x_j\in N$ for $j=1,\dots, n$ and $x_{n+1}\in E_N(m_{n+1}^*m_{n+1})N$.  Other important properties of a Pimsner-Popa basis are the following:
\begin{itemize}
    \item $\sum_{j=1}^{n+1}m_je_Nm_j^*=1$.
    \item $\sum_{j=1}^nm_jm_j^*=[M:N]$.
\end{itemize}
Pimsner-Popa bases are unique in a sense made precise in \cite{pimsnerpopa}.  Relevant for us is how they can be constructed:  Let $g_1,\ldots,g_{n+1}$ be orthogonal projections in $M_1$ with $tr(g_i)=[M:N]^{-1}$ for $i=1,\ldots,n$ and $tr(g_{n+1})=1-n[M:N]^{-1}$.  For $j=1,\ldots,n+1$, take partial isometries $v_j\in M_1$ such that $v_jv_j^*=g_j$, $v_j^*v_j=e_N$ for $j=1,\ldots,n$, and $v_{n+1}^*v_{n+1}\leq e_N$.  
A Pimsner-Popa basis is then obtained by letting $m_j$ be the unique element of $M$ such that $v_j=v_je_N=m_je_N$.

Moving forward, if the inclusion $N\subseteq M$ is clear from  context, we simply write $e$ for the Jones projection $e_N$.

Although we will not study computability of standard invariants in this article, the notion of the standard invariant will arise in examples below.
Jones' \textbf{standard invariant} consists of:
\begin{itemize}
\item 
The two towers of higher relative commutants/centralizer algebras
$$(M_0'\cap M_n)_{n\geq 0} \text{
and }
(M_1'\cap M_{n+1})_{n\geq 0},$$
which are finite dimensional by \cite{jones};
\item 
the Jones projections $e_1 \in N'\cap M_1$, $e_2\in M'\cap M_2\subset N'\cap M_2$, $e_3\in M_1'\cap M_3\subset M'\cap M_3\subset N'\cap M_3$, etc., which satisfy the Temperley-Lieb-Jones relations (see Example \ref{ex:TLJfactor} below);
and
\item 
the Markov trace $\tr$ on $M_\infty=\varinjlim M_n$ restricted to the two towers of centralizer algebras.
\end{itemize}
The standard invariant, and thus the subfactor itself, is called \textbf{finite depth} if there is a global bound on the dimensions of the centers of the higher relative commutants.

\section{The induced presentation on the Jones basic construction}

In this short subsection, we describe how a presentation $M^\#$ of the ambient factor $M$ of a subfactor $N\subseteq M$ naturally induces a presentation $M_1^\#$ of the Jones basic construction $M_1$.  We then present a few lemmas that will be used in the next section when proving the main theorem of this paper.

\begin{defn}
Fix a subfactor $N\subseteq M$ and a presentation $M^\#$ of $M$. The \textbf{induced presentation} $M_1^\#$ of $M_1$ has as its special points the special points of $M^\#$ as well as the element $e$.   
\end{defn}

Note that $M_1^\#$ is indeed a presentation of $M_1$ as the von Neumann algebra generated by the special points is contained in $M_1$ and contains both $e$ and the elements from $M$.

\begin{defn}\label{inducedpresentation}
Given a subfactor $N\subseteq M$ and a presentation $M^\#$ of $M$, the \textbf{induced presentation of $N$}, denoted $N^\#$, is the presentation of $N$ whose special points are those of the form $E_N(x)$, where $x$ is either a special point of $M^\#$ or a rational point of $M^\#$ with $x^\flat < 1$.
\end{defn}

Note that $N^\#$ is indeed a presentation of $N$ as any point in the unit ball of $N$ is the $2$-norm limit of rational points from $M^\#$ with $\flat < 1$ and $E_N$ is continuous with respect to the $2$-norm.




    

\begin{defn}
    We say that a rational point of $M_1^\#$ is in \textbf{syntactic normal form} if it is of the form $a+\sum_i b_iec_i$, where $a$ is a rational point of $M^\#$, each $b_i$ is a rational point of $N^\#$, and each $c_i$ is a rational point of $M^\#$.
\end{defn}


\begin{lem}\label{normalform}
There is a computable function which takes as input a rational point of $M_1^\#$ and returns an equivalent rational point in syntactic normal form.
\end{lem}
\begin{proof}
For each term, apply $exe = E_N(x)e$ ($x \in M$) from the left as necessary until at most a single factor of $e$ remains. Use the linearity of $E_N$ and $\flat$ to write the factors on the left-hand side of $e$ as a rational point of $N^\#$.
\end{proof}

\begin{lem}\label{computablerep}
Suppose $M^\#$ is a $\d$-computable presentation, $[M:N]<\infty$ is $\d$-computable, and $E_N:M^\#\to M^\#$ is a $\d$-computable map. 
Then there is a $\d$-computable algorithm (uniform in codes for $M^\#$, $[M:N]$, and $E_N$) which, when given a $\d$-computable point $v$ of $M_1^\#$ and $k \in \mathbb{N}$, returns a rational point $w$ of $M^\#$ such that $\tnorm{w - y} < 2^{-k}$, where $y$ is the unique element of $M$ for which $ve_N=ye_N$.
\end{lem}

\begin{proof}
Find a rational point $x$ of $M_1^\#$ such that $\tnorm{x-v}^2 < 2^{-2k-2}[M : N]^{-1}$.  By Lemma \ref{normalform}, we can write $x$ in syntactic normal form as $a+\sum_i b_iec_i$. 
Setting $z := a + \sum_i b_iE_N(c_i)$, observe $ze = xe$ and $z$ is a $\d$-computable point of $M^\#$.
Find a rational point $w$ of $M^\#$ such that $\tnorm{w - z} < 2^{-k-1}$.
Then 
\begin{align*}
\tnorm{z - y}^2 &= [M : N]\tnorm{(z - y)e}^2 \\
&= [M : N]\tnorm{xe - ve}^2 \\
&\leq [M : N]\tnorm{x - v}^2.
\end{align*}

Hence $\tnorm{w - y} \leq \tnorm{w - z} + \tnorm{z - y} < 2^{-k}$.

\end{proof}

\section{The main theorem}

In this section, we state and prove the main result of this paper.  Before doing so, we introduce some important terminology.





\begin{defn}
Suppose that $[M:N]<\infty$ and let $m_1,\ldots,m_{n+1}$ be a Pimsner-Popa basis of $M$ over $N$.  Let $M^\#$ be a presentation of $M$. We say that $m_1,\ldots,m_{n+1}$ is a 
\textbf{$\d$-computable Pimsner-Popa basis of $M^\#$ over $N$} if
\begin{enumerate}
    \item each $m_j$ is a $\bf D$-computable point of $M^\#$,
    \item each $E_N(m_j)$ is a $\bf D$-computable point of $M^\#$, and
    \item $E_N(m_{n+1}^*m_{n+1})$ is a $\bf D$-computable point of $M^\#$.
\end{enumerate}
\end{defn}

\begin{remark}
This definition formalizes the properties of a particular Pimsner-Popa basis used in \cite[Proposition 2.9]{gold}.  However, that proof was incorrect as it did not include the condition (3) in the previous definition.  We remedy this error in the proof of our main theorem below.
\end{remark}

\begin{defn}
Suppose that $N$ is a subfactor of $M$, and $N$ and $M$ are equipped with presentations $N^\dagger$ and $M^\#$ respectively.  
We say that $N^\dagger$ is \textbf{$\d$-computably embedded} in $M^\#$ if $M^\#$ is $\d$-computable and the inclusion map $\iota:N^\dagger\hookrightarrow M^\#$ is $\d$-computable (whence $N^\dagger$ is $\d$-computable).
\end{defn}

Recalling the definition of the induced presentation of a subfactor (Definition \ref{inducedpresentation}), the following lemma is clear:

\begin{lem}\label{condexp_pair_equiv}
If $M^\#$ is $\d$-computable, then $E_N:M^\#\to M^\#$ is $\d$-computable if and only if $N^\#$ is $\d$-computably embedded in $M^\#$.
\end{lem}

We can now state our main theorem:

\begin{thm}\label{maintheorem}
Suppose that $N$ is a subfactor of $M$ such that $[M:N]<\infty$.  Fix a presentation $M^\#$ of $M$ and let $M_1^\#$ be the induced presentation of the Jones basic construction $M_1$.  Let $\bf D$ be a Turing oracle such that $M^\#$ is $\bf D$-computable.  Then the following are equivalent:
\begin{enumerate}
    \item $M_1^\#$ is $\bf D$-computable.
    \item The following two statements hold:
    \begin{enumerate}
        \item $E_N:M^\#\to M^\#$ is $\bf D$-computable.
        \item $[M:N]$ is a $\bf D$-computable real number.
        \end{enumerate}
    \item The following two statements hold:
    \begin{enumerate}
    \item There is a presentation $N^\dagger$ of $N$ which is $\d$-computably embedded in $M^\#$.
        \item There is a $\d$-computable Pimsner-Popa basis of $M^\#$ over $N$.
    \end{enumerate}
\end{enumerate}
\end{thm}

\begin{proof}

We first prove the equivalence of (1) and (2).
    Suppose that $M_1^\#$ is $\bf D$-computable.  Since $e$ is a special point of $M_1^\#$, we have $[M:N]^{-1}=\tr(e)=\|e\|_2^2$ is $\bf D$-computable, and thus so is $[M:N]$.  
    Furthermore, $E_N:M^\#\to M^\#$ is $\d$-computable: given a rational point $x$ of $M^\#$ and $k \in \mathbb{N}$, search for a rational point $y$ of $M^\#$ such that $\tnorm{ye-exe}^2 < [M : N]^{-1}2^{-2k}$, in which case \[\tnorm{y-E_N(x)}^2=[M:N]\tnorm{(y-E_N(x))e}^2=[M:N]\tnorm{ye-exe}^2 < 2^{-2k}.\]

Conversely, suppose that $E_N:M^\#\to M^\#$ is $\bf D$-computable and $[M:N]$ is $\bf D$-computable.
Given a rational point $x$ of $M_1^\#$, by Lemma \ref{normalform} we can write $x^*x$ in syntactic normal form as $a + \sum_i b_iec_i$. 
Then \[\tnorm{x}^2 = \tr(x^*x) = \tr(a) + \sum_i \tr(b_iec_i) = \tr(a) + [M : N]^{-1}\sum_i \tr(b_ic_i),\] which is $\d$-computable by Corollary \ref{tracelemma} and Lemma \ref{condexp_pair_equiv}. 

We now prove that (1) implies (3). Suppose (1), equivalently (2), holds.
Trivially then, $N^\#$ is $\d$-computably embedded in $M^\#$ by Lemma \ref{condexp_pair_equiv}.
To finish proving (3), we apply Lemma \ref{orthproj} to find $\d$-computable projections $g_1,\ldots,g_{n+1}$ of $M_1^\#$ with $$\tr(g_i)=[M:N]^{-1} \text{ for } i=1,\ldots,n, \quad \tr(g_{n+1})=1-n[M:N]^{-1}.$$  By Lemma \ref{partialiso}, we effectively find $\d$-computable partial isometries $v_1,\ldots,v_{n+1}$ of $M_1^\#$ such that $v_jv_j^*=g_j$ for $j=1,\ldots,n+1$, $v_j^*v_j=e$ for $j=1,\ldots,n$, and $v_{n+1}^*v_{n+1}\leq e$.
By Lemma \ref{computablerep}, we can effectively find $\d$-computable points $m_1,\ldots,m_{n+1}$ of $M^\#$, where each $m_j$ is the unique element of $M$ that satisfies $v_je = m_je$.
It follows that $m_1,\ldots,m_{n+1}$ is a $\d$-computable Pimsner-Popa basis of $M^\#$ over $N$.

Finally, we prove that (3) implies (2).  Let $N^\dagger$ be $\d$-computably embedded in $M^\#$ and let $m_1,\ldots,m_{n+1}$ be a $\d$-computable Pimsner-Popa basis of $M^\#$ over $N$.  The statement that $E_N:M^\#\to M^\#$ is $\bf D$-computable is \cite[Proposition 2.9]{gold}, but the argument there is incorrect as it is missing the third item in our definition of a $\d$-computable Pimsner Popa basis of $M^\#$ over $N$.  For that reason, we repeat the argument again. Fix a rational point $x$ of $M^\#$ and $k \in \mathbb{N}$.
By our assumptions, we can find rational points $y_1,\ldots,y_{n+1}$ of $N^\dagger$ such that $\tnorm{x - y} < 2^{-k}$, where $y = \sum_{j=1}^n m_jy_j + m_{n+1}E_N(m_{n+1}^*m_{n+1})y_{n+1}$.
Then $$\tnorm{E_N(x) - E_N(y)} \leq \tnorm{x - y} < 2^{-k},$$ where \[E_N(y) = \sum_{j=1}^n E_N(m_j)y_j + E_N(m_{n+1})E_N(m_{n+1}^*m_{n+1})y_{n+1}\] is a $\d$-computable point of $M^\#$, uniformly in $y$.
Thus $E_N$ is $\d$-computable.
It remains to show that $[M:N]$ is $\d$-computable.  However, this follows from the fact that $[M : N] = \tr(\sum_{j=1}^{n+1}m_jm_j^*)$. 
\end{proof}

\begin{defn}
Suppose that $N$ is a subfactor of $M$ with $[M : N] < \infty$ and $M^\#$ is a $\d$-computable presentation of $M$.  We say that $N$ is a \textbf{$\d$-computable subfactor} of $M^\#$ if the equivalent conditions of Theorem \ref{maintheorem} hold.
\end{defn}

\begin{remark}
    Note that statement (b) in item (2) of the main theorem automatically holds if $[M:N]\in [1,4]$ (where it equals the computable real number $4\cos^2(\frac{\pi}{n})$ for some integer $n\geq 3$) or if $N\subset M$ has finite depth, since then it is a cyclotomic integer \cite{ENO05}.
\end{remark}

\begin{example}
By Theorem \ref{maintheorem}, if $N \subseteq M$ is a subfactor and $[M : N]$ is not $\d$-computable, then $N$ is not a $\d$-computable subfactor of $M^\dagger$ for any presentation $M^\dagger$ of $M$.
In particular, $\cal R$ has subfactors of any possible index, but only for countably many of those subfactors is there a presentation $\cal{R}^\dagger$ of $\cal{R}$ for which it is a $\d$-computable subfactor.
On the other hand, it is known there are uncountably many nonisomorphic subfactors $(N_\alpha)_{\alpha<2{^\omega}}$ of $\cal R$ of index $6$ \cite{bnp07}.\footnote{More is true:
the non-isomorphic subfactors of index 6 from \cite{bnp07} all share the same standard invariant.
The article \cite{bv15} constructed a family of subfactors at index 6, all with the same standard invariant $A_3*D_4$, which cannot be classified by countable structures.
}
If we consider the standard presentation $\cal{R}^\#$ of $\cal{R}$, which is computable, then only countably many of the conditional expectation maps $E_{N_\alpha}:\cal R^\#\to \cal R^\#$ can be $\d$-computable, so for all but countably many $\alpha < 2^\omega$, the subfactor $N_\alpha$ is not a $\d$-computable subfactor of $\cal{R}^\#$ even though its index is computable.
\end{example}


We can iterate Theorem \ref{maintheorem} to conclude the $\bf D$-computability of all the factors appearing in the Jones tower from the $\bf D$-computability of the basic construction.  In what follows, we recursively equip the iterated basic constructions $M_n$ with their canonical induced presentations $M_n^\#$.

\begin{cor}
Suppose that $N\subseteq M^\#$ is a $\d$-computable subfactor.  Then $M_n\subseteq M_{n+1}^\#$ is a $\d$-computable subfactor for all $n\geq 0$.  In other words, if $M_1^\#$ is $\bf D$-computable, then $M_n^\#$ is $\bf D$-computable for all $n\geq 1$.
\end{cor}

\begin{proof}
Suppose $M_{n-1}$ is a $\d$-computable subfactor of $M_n^\#$ for some $n \geq 0$, where we set $M_{-1} := N$.
We show $M_n$ is a $\d$-computable subfactor of $M_{n+1}^\#$.
By Theorem \ref{maintheorem} and the fact that $[M_{n+1}:M_n]=[M:N]$ is $\d$-computable, it suffices to show that $E_{M_n}:M_{n+1}^\#\to M_{n+1}^\#$ is $\d$-computable. 
By Lemma \ref{normalform}, we can write any rational point $x$ of $M_{n+1}^\#$ in syntactic normal form as $a + \sum_i b_iec_i$, where $a$ and each $c_i$ are rational points of $M_n^\#$ and each $b_i$ is a rational point of $M_{n-1}^\#$, the presentation on $M_{n-1}$ induced by $M_n^\#$.
Then $E_{M_n}(x) = a + [M : N]^{-1}\sum_i b_ic_i$, which is a $\d$-computable point of $M_{n+1}^\#$, uniformly in $x$, as $M_n^\#$ and $E_{M_{n-1}} : M_n^\# \to M_n^\#$ are $\d$-computable.
\end{proof}

\begin{remark}
If $n \geq 1$, then $M_n$ has two potential induced presentations: one induced by $M_{n-1}^\#$ from the Jones basic construction and one induced by $M_{n+1}^\#$ as its subfactor.
While these presentations are not the same in general, they are $\d$-computably isomorphic when $M_{n+1}^\#$ and $E_{M_n} : M_{n+1}^\# \to M_{n+1}^\#$ are $\d$-computable.
\end{remark}

\section{Improvements to the main result?}

It becomes natural to wonder if any of the conditions in the main theorem are redundant.  While we are unable to accomplish this at the moment, we do mention some partial results in this regard.

Given item (2) in Theorem \ref{maintheorem} above, one may ask whether or not $E_N:M^\#\to M^\#$ being $\bf D$-computable automatically implies that $[M:N]$ must be $\bf D$-computable.  This is the best we are able to do at this point:\footnote{Recall that $\bf D'$ is the \textbf{Turing jump} of $\bf D$, which is the Turing degree which can compute the halting problem relative to $\bf D$.}

\begin{prop}\label{indexprop}
If $E_N:M^\#\to M^\#$ is $\bf D$-computable, then $[M:N]$ is $\bf D'$-computable, whence $M_1^\#$ is $\d'$-computable and a $\d'$-computable Pimsner-Popa basis of $M^\#$ over $N$ exists.
\end{prop}

\begin{proof}
By the equation $[M : N]^{-1} = \inf\big\{\tnorm{E_N(x)}^2 \tnorm{x}^{-2} : x > 0\big\}$, we can compute a decreasing sequence which converges to $[M : N]^{-1}$ from $\d$.
Thus $[M : N]^{-1}$ is $\d'$-computable.
\end{proof}

\begin{question}
If $E_N:M^\#\to M^\#$ is $\d$-computable, must $[M:N]$ be $\bf D$-computable?  In other words, is $M_1^\#$ $\d$-computable if and only if $E_N:M^\#\to M^\#$ is $\d$-computable?
\end{question}

One might also wonder if $N^\dagger$ being $\d$-computably embedded in $M^\#$ for some presentation $N^\dagger$ of $N$ implies that $N^\#$ is $\d$-computably embedded in $M^\#$, that is, if $E_N:M^\#\to M^\#$ is $\d$-computable.  In general, the following is always true:

\begin{prop}\label{jumpexpectation}
If $N^\dagger$ is $\d$-computably embedded in $M^\#$, then $E_N:M^\#\to N^\dagger$ (and hence $E_N:M^\#\to M^\#$) is $\d'$-computable.
\end{prop}
\begin{proof}
Given a rational point $x$ of $M^\#$ and $k \in \mathbb{N}$, search for a rational point $y$ of $N^\dagger$ such that $|\tr((x-y)z) | < 2^{-2k}$ for all rational points $z$ of $N^\dagger$; this can be done effectively with $\d'$. 
Let $z = (E_N(x) - y)^*$ and observe \begin{align*}
\tnorm{E_N(x) - y}^2 &= \tr((E_N(x) - y)z) \\
&= \tr((E_N(x) - x)z) + \tr((x - y)z) \\
&= \tr((x - y)z) \\
&< 2^{-2k}.
\end{align*}
 \end{proof}

 \begin{remark}
The main result of \cite{gold} shows that, under certain extra assumptions, the fact that $N^\dagger$ is $\d$-computably embedded in $M^\#$ implies that $E_N:M^\#\to N^\dagger$ is $\d$-computable.  The context of that paper is quite different as the extra hypotheses there imply that the index is necessarily infinite.
 \end{remark}

Propositions \ref{indexprop} and \ref{jumpexpectation} immediately imply:

\begin{cor}
If $N^\dagger$ is $\d$-computably embedded in $M^\#$, then $[M:N]$ is always $\d''$-computable (whence $M_1^\#$ is $\d''$-computable and a $\d''$-computable Pimsner-Popa basis of $M^\#$ over $N$ exists).
\end{cor}

\section{Examples}

In this final section, we mention a few examples illustrating our main result.

\begin{example}
Let $N$ be any II$_1$ factor and let $M:=M_n(N)=N\bar\otimes M_n(\bb C)$.  Fix a presentation $N^\#$ of $N$ and let $M^\#$ be the presentation of $M$ given by matrices of rational points of $N^\#$.  Note that if $N^\#$ is $\d$-computable, then so is $M^\#$.  Then $[M:N]=n^2$ is clearly computable and the conditional expectation map $E_N:M^\#\to M^\#$ mapping a matrix to its diagonal part is clearly $\d$-computable.  Thus $N$ is a $\d$-computable subfactor of $M^\#$.
\end{example}

\begin{example}
Suppose that $N$ is a II$_1$ factor and $G$ is a finite group.  Consider an action $G\curvearrowright^\alpha N$ by outer automorphisms.  Set $M:=N\rtimes_\alpha G$.  Fix a presentation $N^\#$ of $N$ and suppose that $N^\#$ is $\d$-computable.  Let $M^\#$ be the presentation of $M$ obtained by adding as special points the canonical unitaries $u_g$ for $g\in G$. If, in addition, each $\alpha_g$ is a $\d$-computable automorphism of $N^\#$, then the presentation $M^\#$ is $\d$-computable.  Note also that the index $[M:N]=|G|$ is clearly computable and the conditional expectation map $E_N:M^\#\to M^\#$ mapping a sum $\sum_{g\in G}x_gu_g$ to $x_e$ is $\d$-computable.  Consequently, $N$ is a $\d$-computable subfactor of $M^\#$.
\end{example}

\begin{example}
Suppose once again that we have an action $G\curvearrowright^\alpha M$ of a finite group by outer automorphisms.  Let $M^G$ be the fixed point subalgebra.  The conditional expectation $E_{M^G}:M\to M^G$ is given by $E_{M^G}(x):=\frac{1}{|G|}\sum_{g\in G}\alpha_g(x)$.  The index $[M:M^G]=|G|$ is clearly computable.  Under what conditions is $E_{M^G}:M^\#\to M^\#$ $\d$-computable (that is, when is $M^G$ a $\d$-computable subfactor of $M^\#$)?  This is indeed the case when each $\alpha_g:M^\#\to M^\#$ is a $\d$-computable automorphism.  Is the $\d$-computability of each $\alpha_g$ necessary?  For example, if $G=\bb Z_2$ and $\alpha$ is the unique nontrivial automorphism of $M$, then the $\d$-computability of $E_{M^G}:M^\#\to M^\#$ clearly implies the $\d$-computability of $\alpha:M^\#\to M^\#$.

In regards to this question, we recall that, in the current context, $M_1\cong M\rtimes G$, where the isomorphism is given by mapping $M$ identically to itself and mapping $e_N$ to the projection $e:=\frac{1}{|G|}\sum_{g\in G}u_g$.  Moreover, this map is computable when $M_1$ is equipped with its induced presentation and when $M\rtimes G$ is equipped with the presentation obtained by adding the above projection $e$ to the presentation of $M$.  Note that this presentation differs from the presentation of $M\rtimes G$ described in the previous example.  Assuming that $E_{M^G}$ is $M^\#$-computable, it appears that the condition that each $\alpha_g$ be $M^\#$-computable is equivalent to the above two presentations of $M\rtimes G$ being computably isomorphic. 

\end{example}

\begin{example}
\label{ex:TLJfactor}
We consider the construction of subfactors of $\cal R$ due to Jones \cite{jones}.  Fix a tracial von Neumann algebra $M$ and suppose that $(e_i)_{i\geq 1}$ are projections in $M$ satisfying:
\begin{enumerate}
    \item $e_ie_{i\pm 1}e_i=t e_i$ for some $t\leq 1$.
    \item $e_ie_j=e_je_i$ if $|i-j|\geq 2$.
    \item $\tr(we_i)=t\tr(w)$ for any word $w$ on $1,e_1,\ldots,e_{i-1}$.
\end{enumerate}
Letting $P:=\{e_i \ : \ i\geq 1\}''$ be the von Neumann subalgebra of $M$ generated by the $e_i$'s for $i\geq 1$ and $P_t:=\{e_i \ : \ i\geq 2\}''$ denote the subalgebra of $P$ generated by the $e_i$'s for $i\geq 2$, we have that $P\cong \cal R$ and that $P_t$ is a subfactor of $P$ of index $[P:P_t]:=\frac{1}{t}$.  Moreover, this construction is only possible if $t\leq \frac{1}{4}$ or if $t=\frac{1}{4}\sec^2(\frac{\pi}{n})$ for $n=3,4,5\ldots$

We equip $P$ and $P_t$ with the presentations $P^\#$ and $P_t\#$ consisting of the set of the appropriate set of $e_i$'s.  It is immediate from \cite[Lemma 4.1.6]{jones} that $P^\#$ and $P_t^\#$ are $\d$-computable, where $\d$ is any oracle for which $t$ is $\d$-computable.

We claim that, for any such $\d$, the subfactor $P_t$ is a $\d$-computable subfactor of $P^\#$.  Since $t$ is $\d$-computable and $[P:P_t]=\frac{1}{t}$, we have that $[P:P_t]$ is $\d$-computable.  It remains to see that $E_{P_t}:P^\#\to P^\#$ is $\d$-computable, which follows from the fact that $E_{P_t}(e_1)=t\cdot 1$.  
\end{example}

\end{document}